\documentclass[11pt]{amsart}
 \usepackage{amscd,amsmath,amsthm,amssymb}
 \usepackage{pstcol,pst-plot,pst-3d}%\usepackage[T1]{fontenc}
 \usepackage{graphicx}
 \psset{unit=0.7cm,linewidth=0.8pt,arrowsize=2.5pt 4}
\newpsstyle{fatline}{linewidth=1.5pt}
\newpsstyle{fyp}{fillstyle=solid,fillcolor=verylight}
\definecolor{verylight}{gray}{0.97}
\definecolor{light}{gray}{0.93}
\definecolor{medium}{gray}{0.82}
 \def\opn#1#2{\def#1{\operatorname{#2}}} % to make operators
 \opn\chara{char} \opn\length{\ell} \opn\pd{pd} \opn\rk{rk}
 \opn\projdim{proj\,dim} \opn\injdim{inj\,dim} \opn\rank{rank}
 \opn\depth{depth} \opn\grade{grade} \opn\height{height}
 \opn\embdim{emb\,dim} \opn\codim{codim}
 \opn\hreg{hreg}

 \opn\Tr{Tr} \opn\bigrank{big\,rank}
 \opn\superheight{superheight}\opn\lcm{lcm}
 \opn\trdeg{tr\,deg}%\emph{
 \opn\reg{reg} \opn\lreg{lreg} \opn\ini{in} \opn\lpd{lpd}
 \opn\size{size} \opn\sdepth{sdepth}
 \opn\link{link}\opn\fdepth{fdepth}\opn\lex{lex}
 \opn\div{div} \opn\Div{Div} \opn\cl{cl} \opn\Cl{Cl}
 \opn\Spec{Spec} \opn\Supp{Supp} \opn\supp{supp} \opn\Sing{Sing}
 \opn\Ass{Ass} \opn\Min{Min}\opn\Mon{Mon}
 \opn\Ann{Ann} \opn\Rad{Rad} \opn\Soc{Soc}
 \opn\Im{Im} \opn\Ker{Ker} \opn\Coker{Coker} \opn\Am{Am}
 \opn\Hom{Hom} \opn\Tor{Tor} \opn\Ext{Ext} \opn\End{End}
 \opn\Aut{Aut} \opn\id{id}
 
 \opn\nat{nat}
 \opn\pff{pf}%   \pf exists already
 \opn\d{d} \opn\GL{GL} \opn\SL{SL} \opn\mod{mod} \opn\ord{ord}
 \opn\Gin{Gin} \opn\Hilb{Hilb}\opn\sort{sort}
 \opn\aff{aff} \opn\a{a}
 \opn\con{conv} \opn\relint{relint} \opn\st{st}
 \opn\lk{lk} \opn\cn{cn} \opn\core{core} \opn\vol{vol}
 \opn\link{link} \opn\star{star}\opn\lex{lex}\opn\set{set}
 \opn\diam{diam}
 \opn\gr{gr}
 
 \def\pot#1#2{#1[\kern-0.28ex[#2]\kern-0.28ex]}

 \opn\dirlim{\underrightarrow{\lim}}
 \opn\inivlim{\underleftarrow{\lim}}
 \def\Implies{\ifmmode\Longrightarrow \else
         \unskip${}\Longrightarrow{}$\ignorespaces\fi}
 \def\implies{\ifmmode\Rightarrow \else
         \unskip${}\Rightarrow{}$\ignorespaces\fi}
 \def\iff{\ifmmode\Longleftrightarrow \else
         \unskip${}\Longleftrightarrow{}$\ignorespaces\fi}

 \let\:=\colon
 \newtheorem{Theorem}{Theorem}[section]
 \newtheorem{Lemma}[Theorem]{Lemma}
 \newtheorem{Corollary}[Theorem]{Corollary}
 \newtheorem{Proposition}[Theorem]{Proposition}
 \newtheorem{Remark}[Theorem]{Remark}
 
 \newtheorem{Example}[Theorem]{Example}
 
 \newtheorem{Definition}[Theorem]{Definition}

 \let\epsilon\varepsilon
 \let\kappa=\varkappa
 \def\qed{\ifhmode\textqed\fi
       \ifmmode\ifinner\quad\qedsymbol\else\dispqed\fi\fi}
 \def\textqed{\unskip\nobreak\penalty50
        \hskip2em\hbox{}\nobreak\hfil\qedsymbol
        \parfillskip=0pt \finalhyphendemerits=0}
 \def\dispqed{\rlap{\qquad\qedsymbol}}
 
 \opn\dis{dis}
 \def\pnt{{\raise0.5mm\hbox{\large\bf.}}}
 
 \opn\Lex{Lex}

\begin{document}
  %\linenumbers

 \title{ Linear syzygy graph and linear resolution}

\author{ Erfan manouchehri}
\address{ Erfan manouchehri,
Department of Mathematics, University of Kurdistan,
P.O.Box:66177-15175, Sanadaj, Iran}
\email{erfanm6790@yahoo.com}

\author{Ali Soleyman Jahan}
\address{Ali Soleyman Jahan, Ali Soleyman Jahan,
Department of Mathematics, University of Kurdistan,
P.O.Box:66177-15175, Sanadaj, Iran}
\email{A.solaimanjahan@uok.ac.ir}

\date{}
\begin{abstract}
For each squarefree monomial ideal $I\subset S = k[x_{1},\ldots, x_{n}] $, we associate a simple graph $G_I$ by using the first
linear syzygies of $I$. In cases, where $G_I$ is a cycle
or a tree, we show the following are equivalent:\\$(a)$ \;$I$ has a linear resolution;\\
$(b)$\; $I$ has linear quotients;\\$(c)$\; $I$ is variable-decomposable.\\  In addition,
with the same assumption on $G_I$, we characterize all monomial ideals
with a linear resolution. Using our results, we characterize all Cohen-Macaulay
codimension $2$ monomial ideals with a linear resolution. As an other application of our results, we also characterize  all Cohen-Macaulay simplicail complexes in cases that $G_{\Delta}\cong G_{I_{\Delta^{\vee}}}$ is  a cycle or a tree.

\end{abstract}

 \maketitle

\section*{Introduction}
\label{intro}

Let $S = k[x_{1},\ldots,x_{n}] $ be the polynomial ring in $n$
variables over a field $k$ and $I$ be a monomial ideal in $S$. We
say that $I$ has a $d$-linear resolution if the graded minimal free
resolution of $I$ is of the form: $$ 0 \longrightarrow
S(-d-p)^{\beta_p} \cdots \longrightarrow S(-d-1)^{\beta_1}
\longrightarrow S(-d)^{\beta_0} \longrightarrow I \longrightarrow 0.
$$ In general it is not  easy to find ideals with linear
resolution. Note that the free resolution of a monomial ideal and,
hence, its linearity depends in general on the characteristic of the
base field.

Let $I \subseteq S$ be a monomial ideal. We denote by $G(I)$ the
unique minimal monomial set of generators of $I$. We say that $I$
has linear quotients if there exists an order $\sigma = u_1,
\ldots, u_m$ of $G(I)$ such that the colon ideal $<u_1, \ldots,
u_{i-1}> : u_i$ is generated by a subset of the variables, for $i =
2, \ldots, m$. Any order of the generators for which, $I$ has linear
quotients, will be called an admissible order. Ideals with linear
quotients were introduced by Herzog and Takayama \cite{HT}.   Note
that linear quotients is  purely
combinatorial property of an ideal $I$ and, hence, does not depend on the
characteristic of the base field. Suppose that $I$ is a graded ideal
generated in degree $d$. It is known that if $I$ has linear
quotients, then $I$ has a $d$-linear resolution \cite[Proposition
8.2.1]{HH}.

The concept of variable-decomposable monomial ideal was first
introduced by Rahmati and Yassemi \cite{RY} as a dual concept of
 vertex-decomposable simplicial complexes. In case that $I=I_{\Delta^{\vee}}$, they proved that $I$
 is variable-decomposable if and only if $\Delta$ is vertex-decomposable. Also they proved   if a
 monomial ideal $I$ is variable-decomposable, then it has linear
 quotients. Hence for monomial ideal generated in one degree, we
 have the following implications:\\
 $I$ is variable-decomposable $\Longrightarrow \; I$ has linear
 quotients $\Longrightarrow \; I$ has a linear resolution.\\
 However, there are ideals with linear resolution but
without linear quotients, see \cite{CH}, and  ideals with
linear quotients which are  not variable-decomposable, see \cite[Example 2.24]{RY}.

The problem of existing $2$-linear resolution is completely solved
by Fr\"{o}berg \cite{Fro} (See also \cite{MO}). Any ideal of $S$
which is generated by squarefree monomials of degree $2$ can be
assumed as edge ideal of a simple graph. Fr\"{o}berg proved that the
edge ideal of a finite simple graph $G$ has a linear resolution if and
only if the complementary graph $\bar{G}$ of $G$ is chordal. Trying to generalize the result of Fr\"{o}berg for monomial
ideals generated in degree $d$, $d\geq 3$, is an interesting
problem on which several mathematicians including E. Emtander \cite{E} and
R.Woodroofe \cite{W}
have worked.

It is known that monomial ideals with 2-linear resolution have
linear quotients \cite{HeHiZ}.  Let $I=I_{\Delta^{\vee}}$ be a
squarefree monomial ideal generated in degree $d$ which  has a
linear resolution. By a result of Eagon-Reiner \cite{ER}, we know
 $\Delta$ is a Cohen-Macaulay of dimension $n-d$ . In \cite{AS}
Soleyman Jahan and Ajdani proved  if $\Delta$ is a
Cohen-Macaulay simplicial complex of codimension $2$, then $\Delta$
is vertex-decomposable. Hence, by  \cite[Theorem 2.10]{RY},
$I_{\Delta^{\vee}}$  is a variable-decomposable monomial ideal
generated in degree $2$.  Therefore  if $I=I(G)$ is the edge ideal of a simple graph $G$, then  the following are equivalent:
\begin{itemize}
\item[(a)] $I$ has a linear resolution;
\item[(b)] $I$ has linear quotients;
\item[(c)] $I$ is variable-decomposable ideal.
\end{itemize}

 So it is natural to look for some other classes of
monomial ideals with the same property.

The paper proceeds as follows.  In Section \ref{sec:1}, we associated a simple graph $G_I$ to a squarefree monomial ideal $I$ generated in
degree $d\geq 2$. In Theorem
\ref{02}, we show that if $G_I\cong C_m$, $m\geq 4$, then $I$ has a
linear resolution if and only if it has linear quotients and it is equivalent
to $I$ is a variable-decomposable. With the same assumption on $G_I$, we characterize all  monomial ideals with a linear resolution.

In Section \ref{sec:2}, we
consider monomial ideal $I$ where $G_I$ is a tree. We prove that if $I$
has linear relations, then $G_I$ is a tree if and only if
$\projdim(I)=1$ (see Theorem \ref{01}).  In Theorem \ref{tree} we show that if $G_I$ is a tree, then the following are equivalent:
\begin{itemize}
\item[(a)]
$I$ has a linear resolution;
\item[(b)]
$I$ has  linear relations;
\item[(c)]
$G_I^{(u,v)}$ is a connected graph for all $u$, and $v$ in $G(I)$;
\item[(d)]
If $u=u_1,u_2,\ldots,u_s=v$ is the unique path  between $u$ and $v$ in $G_I$, then $F(u_j)\subset F(u_i)\cup F(u_k)$  for all $1\leq i\leq j\leq k \leq s$;
\item[(e)]
 $L$ has a linear  resolution for all $L \subseteq I$, where $G(L)\subset G(I)$ and $G_L$ is a line
\end{itemize}
In addition, it is shown
that $I$ has a linear resolution if and only if it has linear
quotients and if and only if it is variable-decomposable, provided that
$G_I$ is a tree (see Theorem \ref{019}).

Let $\Delta_I$ be the Scarf
complex of $I$. In Theorem \ref {011} we prove that in the case that $G_I$ is a tree, $I$
has a linear resolution if and only if $G_I\cong \Delta_I$.

In Section \ref{sec:3}, as applications of our results in Corollary \ref{cm2}, we characterize all Cohen-Macaulay monomial ideals of codimension $2$ with a linear resolution.  Let  $t\geq 2$ and $I_t(C_n)$ ($I_t(L_n)$) be the path ideal of length $t$ for $n$-cycle $C_n$ ( $n$-line $L_n$). We show that $I_t(C_n)$ ($I_t(L_n)$ has a linear resolution if and only if $t=n-2$ or $t=n-1$ ($t\geq n/2$), see Corollary \ref{cn} and  Corollary \ref{Ln}.

Finally, we consider simplicial complex $\Delta=\langle F_1,\ldots, F_m\rangle$. It is shown that  $\Delta$ is connected in codimension one if and only if $G_{I_{\Delta^{\vee}}}$ is a connected graph, see Lemma \ref{cd1-c}. In Corollary \ref{delta fg}, we show that $I_{\Delta^{\vee}}$ has linear relations if and only if $\Delta^{(F,G)}$ is connected in codimension one for all facets $F$ and $G$ of $\Delta$. Also, we introduce a simple graph $G_{\Delta}$ on vertex set $\{F_1,\ldots,F_m\}$ which is isomorphic to $G_{I_{\Delta^{\vee}}}$. As Corollaries of our results, we show that  if $G_{\Delta}$ is  a cycle or a tree, then the following are equivalent:
 \begin{enumerate}
 \item[(a)] \; $\Delta$ is Cohen-Macaulay;
 \item[(b)]\; $\Delta$ is pure shellable;
 \item[(c)]\; $ \Delta$ is pure vertex-decomposable.
 \end{enumerate}
 In addition, with the same assumption on $G_{\Delta} $ all Cohen-Macaulay simplicial complexes are characterized.

 Note that for monomial ideal $I=<u_1, \ldots,
u_m>$ and monomial $u$ in $S$, $I$ has a linear resolution (has linear quotients,  is variable-decomposable) if and only if
$uI$ has a linear resolution (has linear quotients,  is variable-decomposable). Hence, without the loss of generality, we
assume that $\gcd(u_i:\; u_i\in G(I))=1$. Also, one can see that a
monomial ideal $I$ has a linear resolution (has linear quotients, is variable-decomposable) if and
only if its polarization has a linear resolution (has linear quotients, is variable-decomposable).
Therefore in this paper we only consider squarefree monomial
ideals.

\section{ monomial ideals whose  $G_I$ is a cycle}
\label{sec:1}
Let $I$ be a monomial ideal which is generated in one degree. First, we  recalling  some definitions and known facts which  will be useful later.

\begin{Proposition}\label{qouresu}\cite[Proposition
8.2.1]{HH} {\em
Suppose $I \subseteq S $ is a monomial ideal generated in degree $d$.
If $I$ has linear quotients, then $I$ has a $d$-linear
resolution.}
\end{Proposition}

Let $u = x_1^{a_1} \ldots x_n^{a_n}$ be a monomial in $S$. Set $F(u):=\{i\; :\; a_i>0\}=\{i\; :\; x_i\mid u\}$. For
another monomial $v$, we set $[u, v] = 1$ if $x_i^{a_i} \nmid v $
for all $i \in F(u)$. Otherwise, we set $[u, v] \neq 1$. For a
 a monomial ideal $I \subseteq S$, set $I_u=<
u_i \in G(I) :\ [u, u_i] = 1>$ and $I^u = < u_j \in G(I) :\ [u, u_j]
\neq 1>$.

\begin{Definition}{\em
Let $I$ be a monomial ideal with
$G(I)=\{u_1, \ldots, u_m \}$. A monomial $u =x_1^{a_1} \ldots x_n^{a_n}$
 is called shedding if $I_u \neq 0$ and for each
$u_i \in G(I_u)$ and $l \in F(u)$, there exists $u_j \in G(I^u)$
such that $u_j:u_i=x_l$.  Monomial ideal  $I$ is  $r$-decomposable  if $m = 1$ or
else has a shedding monomial $u$ with $\mid F(u)\mid \leq r + 1$
such that the ideals $I_u$ and $I^u$ are $r$-decomposable.}
\end{Definition}

A monomial ideal is decomposable if it is $r$-decomposable for some
$r \geq 0$ . A $0$-decomposable ideal is called
variable-decomposable. In \cite{RY} the authors proved the following result:

\begin{Theorem}\label{decqou}{\em
 Let $I$ be a monomial ideal with $ G(I) =
\{u_1, \ldots, u_m \}$. Then $I$ is decomposable if and only if it
has linear quotients.}
\end{Theorem}

Let $I$ be a squarefree monomial  ideal and
\[
F: \; 0 \longrightarrow F_p \cdots \longrightarrow F_1
\longrightarrow F_0 \longrightarrow I \longrightarrow 0
\]
be the minimal graded free $S -$resolution of $I$, where $
F_i=\bigoplus_{j}S(-j)^{\beta_{ij}} $ for all $i$. Set $\varphi :\; F_0 \longrightarrow I$    and  $ \psi :
F_{1}\longrightarrow F_{0}$, where  $\varphi$ maps a basis element $e_i$ of $F_0$ to $u_i\in G(I)$ and $\psi$  maps a basis element $g_i$  of  $F_1$  to an element of a  minimal  generating set of
$\ker(\varphi)$. Monomial ideal  $I$ has {\em linear relations } if $\ker(\varphi)$ is generated minimally by a set of linear forms.

We associate to $I$ a simple
graph $G_{I}$  whose vertices are
labeled by the elements of $G(I)$. Two vertices $u_i$ and $u_j$ are
adjacent if there exist variables $x, y$ such that $xu_i=yu_j$. This
graph was first introduced by Bigdeli, Herzog and Zaare-Nahandi
\cite{BHZ}.

\begin{Remark} {\em  If $I$ is a squarefree monomial ideal, then two type of  $3$-cycle $u_{i_1}, u_{i_2}, u_{i_3}$  may
appear in $G_I$.

 $(i)$:  If $ F(u_{i_1})=A \cup \{j, k\}, F(u_{i_2})=A \cup \{i, k\} $
  and $ F(u_{i_3})=A \cup \{i, j\} $.
Then
 $ x_ie_{i_1}-x_ke_{i_3}=(x_ie_{i_1}-x_je_{i_2})+(x_je_{i_2}-x_ke_{i_3})$.
 In this case one of the linear forms can be written as a linear combination of two other linear forms.

 $(ii)$: If $ F(u_{i_1})=A \cup \{i\}, F(u_{i_2})=A \cup \{j\} $ and
$ F(u_{i_3})=A \cup \{k\} $. In this case  the three
linear forms   are independent.}
\end{Remark}

The number of the minimal generating set of $ \ker(\varphi) $ in degree $d+1$ is $\beta_{1(d+1)}$ and   $\beta_{1(d+1)} \leq \mid E(G_I) \mid$. It is clear that  equality holds if $G_{I}  $ has no  $ C_{3} $ of type $(i)$. If $G_I$ has  a $C_3$ of type $(i)$, then we remove one  edge  of this cycle.
In this way, we obtain a graph  $G_I$ with no $C_3$ of type $(i)$ and called  it the first syzygies graph of $I$.

Our aim is to study minimal free resolution of $I$ via
some combinatorial properties of $G_I$.  Set $
x_{F}:=\prod_{i \in F} x_{i} $ for each $  F\subset [n]=\{1,\ldots,n\} $.
\begin{Remark}\label{001}{\em
 Let $I$ be a squarefree monomial ideal. If  $u_i=x_{F_i}$ and $u_j=x_{F_j}$ are two elements in $G(I)$ such that  $w_i u_i
 =w_ju_j$, then there exists a
 monomial $w\in S$ such that
 $w_i=wx_{F_j \setminus  F_i}$ and $w_j=wx_{F_i\setminus  F_j}$.}
\end{Remark}

\begin{Lemma} \label{002}{\em
Let $I$ be squarefree monomial ideal.   If there is a path of length $t$ between $u$ and $v$ in $G_I$, then
one can obtain  monomials $w_{i}$ and $w_{j}$ from the given path
such that $w_{i}u =w_{j}v$ and $\deg w_{i}$=$\deg w_{j} \leq t.$}
\end{Lemma}
\begin{proof}
We proceed by induction on $t$. The case $t=1$ is obvious. Let
$t=2$  and  $u,w,v$ be a path of length $2$ in $G_I$.
Since $u$ and $w$ and $w$ and $v$ are adjacent, we have $x_{i_1}u=x_{i_2}w$ and
 ${x_i}_{3}w={x_i}_{4}v$. Hence $x_{i_1}x_{i_3}u=x_{i_2}x_{i_4}v$.

Now assume that $t>2$ and  $u=u_{i_0},
u_{i_1}, \ldots, u_{i_{t-1}}, u_{i_t}=v$ is a path of length $t$.
Hence $u=u_{i_0}, u_{i_1}, \ldots, u_{i_{t-1}}$ is a path of length
$t-1$. Using induction hypothesis, we conclude that there are monomials $w^{'}_{i}$ and $w^{'}_{j}$ such that
  $w^{'}_{i}u =w^{'}_{j}u_{i_{t-1}}$, where $\deg w^{'}_{i}=\deg w^{'}_{j} \leq t-1$.
  Since $ v $ and $u_{i_{t-1}}  $ are adjacent, there exist variable $x, y$ such that
  $xu_{i_{t-1}}=yv$.
 Therefore  $ xw^{'}_{i}u =yw^{'}_{j}v$ and
$\deg xw^{'}_{i}=\deg yw^{'}_{j} \leq t$.
 \end{proof}

The following example shows that the inequality $\deg w_{i}$=$\deg
w_{j} \leq k$ can  be pretty strict.

 \begin{Example}\label{1}{\em
  Consider  monomial ideal
 $I=<u, v, w, z>\subset k[x_1,\ldots,x_5]$, where
 $u=x_{1}x_{2}x_{3}$, $w=x_{1}x_{2}x_{4}$, $z=x_{1}x_{4}x_{5}$ and $v=x_{3}x_{4}x_{5}$.
 We have a path of length $3$ between $u$ and $v$, but  $x_{4}x_{5}u=x_{1}x_{2}v.$

 \psset{unit=0.15cm} \begin{pspicture}(5,0)(-35,15)
\put(2,10){$_{\bullet}$} \put(0.2,10){$u$} \put(8,10){$_{\bullet}$}
\put(7.7,11){$w$} \put(14,10){$_{\bullet}$} \put(14,11){$z$}
\put(20,10){$_{\bullet}$} \put(21,10){$v$}
\psline(2.2,10)(8,10)
\psline(14,10)(8,10)
\psline(14,10)(20.2,10)
\end{pspicture}}
\end{Example}

\begin{Lemma} \label{003}{\em
Let $I$ be  squarefree monomial ideal which has  linear relations.
Then $G_{I}$ is a connected graph.}
\end{Lemma}
\begin{proof}
For  any   $u_{i}, u_{j}\in G(I)$, there exist
monomials $w_{i}$ and $w_{j}$ such that $w_{i}u_{i} =w_{j}u_{j}$ and,
hence, $w_{i}e_{i} -w_{j}e_{j} \in \ker(\varphi)$ . Since  $
\ker(\varphi) $ is generated by linear forms one has :
$$w_i e_i -
w_j e_j =f_{i_1}(x_{k_1}e_i-{x_{k_2}}^{'}e_{i_2})+
f_{i_2}(x_{k_2}e_{i_2}-{x_{k_3}}^{'}e_{i_3})+ \ldots+
f_{i_t}(x_{k_t}e_{i_t}-{x_{k_{t+1}}}^{'}e_{j}),$$ where $f_{ij}\in
S$ for $j=0,\ldots, t$. Therefore $u_{i},u_{i_2},\ldots,u_{i_t},
u_{j}$ is a path in $G_I$.
\end{proof}

The following example shows that the converse of Lemma \ref{003} is
not true in general.
\begin{Example} {\em
Consider  monomial ideal\label{2}
 $I=<u, v, w, z, q>\subset k[x_1,\ldots,x_5]$, where
 $u=x_{1}x_{2}x_{3},$ $v=x_{1}x_{2}x_{4},$ $w=x_{1}x_{4}x_{5},$  $z=x_{4}x_{5}x_{6}$ and
  $q=x_{3}x_{5}x_{6}$. It is easy to see that $G_I$ is the following connected
  graph.

 \psset{unit=0.15cm} \begin{pspicture}(20,0)(-35,15)
\put(2,10){$_{\bullet}$} \put(0.2,10){$u$}
\put(7.8,10){$_{\bullet}$} \put(7.7,11){$v$}
\put(13.9,10){$_{\bullet}$} \put(13.6,11){$w$}
\put(19.7,10){$_{\bullet}$} \put(19.7,11){$z$}
\put(25.7,10){$_{\bullet}$} \put(27,10){$q$}
\psline(2.2,10)(7.8,10)
\psline(13.9,10)(7.8,10)
\psline(13.9,10)(19.7,10)\psline(25.9,10)(19.7,10)
\end{pspicture}

 However $I$ has not linear relations. It's minimal free S-resolutions is:
$$0 \longrightarrow S(-6) \longrightarrow S(-4)^4+S(-5) \longrightarrow S(-3)^5 \longrightarrow  I\longrightarrow 0.$$}
\end{Example}

\begin{Remark} \label{4}{\em
Let $I$ be a squarefree monomial  ideal and  $u=u_{i_1}, u_{i_2},
\ldots, u_{i_{t-1}}, u_{i_t}=v$ be  a path in $G_I$. If $ r \in
F(v)$ and $r \notin F(u) $, then $ x_r $ is the coefficient of
some $ e_{i_j}  $ in the linear relations which comes from the given
path.}
\end{Remark}

\begin{Remark}\label{13}{\em
Let $u=u_{i_1}, u_{i_2}, \ldots, u_{i_{t-1}}, u_{i_t}=v$ be a path
in $G_I$. We know  there exist minimal( with respect to divisibility)
monomials $w$ and $w^{'}$  such that  $we_{i_1}
-w^{'}e_{i_t} \in \ker(\varphi)$ and, hence,\\
$w e_{i_1} -
w^{'} e_{i_t} =f_{i_1}(x_{k_1}e_{i_1} - {x_{k_2}}^{'}e_{i_2})+
f_{i_2}(x_{k_2}e_{i_2} - {x_{k_3}}^{'}e_{i_3})+ \ldots+\\
f_{i_{t-1}}(x_{k_{t-1}}e_{i_{t-1}} - {x_{k_{t}}}^{'}e_{i_t}).$

 If  for each $j$, $1\leq j \leq t$, \; $F(u_{i_j})
\subseteq F(u) \cup F(v)$ and $ x_{l}\mid w $,
then $ x_l \nmid w^{'}$. By Remark \ref{4}  $ x_{l} $
is the coefficient of some $ e_{i_r} $ which appear in the above
equation. Hence, there exist $ u_{i_j} $ such that $  l \in
F(u_{i_j}) $. Since $F(u_{i_j}) \subseteq F(u) \cup F(v)$ and $ l
\notin F(u)$, one has $ l \in F(v)$. So $ x_{l}\nmid w^{'} $.
Similarly for arbitrary $ x_{r} $ where $ x_{r}\mid w^{'} $, one
has $ x_{r}\nmid w $. Hence we conclude that $w=x_{F(v) \setminus
F(u)}$ and $w^{'}=x_{F(u) \setminus F(v)}$.}
\end{Remark}

\begin{Remark}\label{16}{\em
Let $w_{i_1}$ and $w_{i_t} $ be two minimal monomials (with respect to divisibility) in $S$ such that
$w_{i_1} e_{i_1} -w_{i_t}e_{i_t}\in\ker(\varphi)$. Assume that
\begin{center}
$w_{i_1} e_{i_1} -w_{i_t}e_{i_t} =f_{i_1}(x_{k_1}e_{i_1} -
{x_{k_2}}^{'}e_{i_2})+ f_{i_2}(x_{k_2}e_{i_2} -
{x_{k_3}}^{'}e_{i_3})+ \ldots+ f_{i_{t-1}}(x_{k_{t-1}}e_{i_{t-1}} -
{x_{k_t}}^{'}e_{i_t})$.
\end{center}
If $x_i \nmid u_{i_1}$ and there exist $u_{i_r}$, $2\leq r\leq t$, such that $x_i \mid
u_{i_r}$, then $x_i \mid w_{i_1}$.
We may assume that $r$ is the smallest number with the property that $x_i\mid u_{i_r}$.
We know that   $ f_{i_{r-2}}(x_{k_{r-2}}e_{i_{r-2}} -
{x_{k_{r-1}}}^{'}e_{i_{r-1}})+ f_{i_{r-1}}(x_ie_{i_{r-1}} -
{x_{k_r}}^{'}e_{i_r})$ is a part of  above equation. Since  in the above equation $e_{i_{r-1}}$ must be eliminated, we have  $f_{i_{r-1}}x_i=f_{i_{r-2}}{x_{k_{r-1}}}^{'}$. Hence, $x_i \mid
f_{i_{r-2}}$. Also, $e_{i_{r-2}}$ must be eliminated and, hence, one has
$f_{i_{r-2}}x_{k_{r-2}}=f_{i_{r-3}}{x_{k_{r-2}}}^{'}$. Therefore  $x_i \mid
f_{i_{r-3}}$. Continuing these procedures yields $x_i \mid
f_{i_{1}}$, i.e $x_i \mid w_{i_1}$.

Similarly if $x_i \nmid u_{i_t}$
and there exist $u_{i_r}$, $1\leq r\leq t-1$, such that $x_i \mid u_{i_r}$, then $x_i
\mid w_{i_t}$.}
\end{Remark}

For all $u, v \in G(I)$, let $G^{(u,v)}_I $ be the induced subgraph
of $G_{I}$ on vertex set
$$V (G^{(u,v)}_I) = \{w \in G(I): F(w)
\subseteq F(u) \cup F(v)\}.$$

The following fact was proved by Bigdeli, Herzog and Zaare-Nahandi
\cite{BHZ}. Here we present  a different proof of it.

\begin{Proposition} \label{009}{\em
Let $I$ be a squarefree monomial ideal which is generated in degree $d$. Then $I$ has
linear relations if and only if $G^{(u,v)}_I $ is connected for all
$u, v \in G(I)$.}
\end{Proposition}
\begin{proof}
Assume that  $I$ has  linear relations and $u, v \in G(I)$. We know
that $x_{F(v) \setminus  F(u)}e_{u} -x_{F(u) \setminus  F(v)}e_{v}
\in \ker(\varphi)$ . Since  $ \ker(\varphi) $ is generated by linear
forms
\begin{center}
$x_{F(v) \setminus  F(u)}e_{u} -x_{F(u) \setminus  F(v)}e_{v}
=f_{i_1}(x_{k_1}e_{i_1} - {x_{k_2}}^{'}e_{i_2})+
f_{i_2}(x_{k_2}e_{i_2} - {x_{k_3}}^{'}e_{i_3})+ \ldots+
f_{i_{t-1}}(x_{k_{t-1}}e_{i_{t-1}} - {x_{k_t}}^{'}e_{t})$.
\end{center}
Hence $u=u_{i_1}, u_{i_2}, \ldots, u_{i_{t-1}}, u_{i_t}=v$ is a path
in $G_I$. Now  it is enough to
show that $F(u_{i_j}) \subseteq F(u_{i_1}) \bigcup F(u_{i_t})$ for all
$i_j$, $1  < j  < t$. Assume to the contrary that there
exist $k$, $1  < k  < t$, such that $F(u_{i_k}) \nsubseteq F(u_{i_1}) \bigcup F(u_{i_t})$. Let $ l \in F(u_{i_k}) $ and $l \notin
F(u_{i_1}) \bigcup F(u_{i_t})$. By Remark $\ref{16}$\;  $x_l \mid
x_{F(v) \setminus  F(u)}$ and $x_l \mid x_{F(u) \setminus  F(v)}$.
This is a contradiction.

For converse, we know that $\ker(\varphi)$ is generated by $x_{F_v \setminus  F_u}e_{u} -x_{F_u
\setminus  F_v}e_{v} $, where $u, v \in G(I)$. By our assumption,  $ G^{(u,v)}_I $ is a
connected graph for all $u, v\in G(I)$. Therefore  there exist  a path $u=u_{i_1}, u_{i_2}, \ldots, u_{i_{t-1}}, u_{i_t}=v$
between $u$ and $v$ in $G^{(u,v)}$. By Remark \ref{13}, one has
$$x_{F(v) \setminus  F(u)}e_{i_1} -
x_{F(u) \setminus F(v)}e_{i_t}=f_{i_1}(x_{k_1}e_{i_1} - {x_{k_2}}^{'}e_{i_2})+ \ldots+
f_{i_{t-1}}(x_{k_{t-1}}e_{i_{t-1}} - {x_{k_t}}^{'}e_{t}).$$

Hence, $x_{F(v) \setminus  F(u)}e_{i_1}- x_{F(v) \setminus  F(u)}e_{i_t}$ is a linear combination of linear
forms.
\end{proof}

\begin{Lemma}\label{004}{\em
Let I be a squarefree monomial ideal. Then one can assign to each
cycle of $G_{I}$ an element in $\ker( \psi)$.}
\end{Lemma}
\begin{proof}
Let $u_{i_1}, u_{i_2}, \ldots, u_{i_{t-1}}, u_{i_t}, u_{i_1}$ be a
cycle in $G_{I}$. Then we have two paths $u_{i_1}, u_{i_2} $ and
$u_{i_2},\ldots, u_{i_t}, u_{i_1} $. Since $\{u_{i_1}, u_{i_2}\} \in
E(G_I)$, there exist variables $x$ and $y$ such that  $ xe_{i_1} -ye_{i_2} \in \ker(\varphi)=\Im( \psi) $. This  is an element in the minimal set of generators
of $\ker(\varphi)$. Hence,  there exist  a basis element $g$  of
$F_1$ such that $\psi(g)=xe_{i_1} -ye_{i_2} $.

By Lemma \ref{002}, there exist monomials $w_{1}$ and $w_{2}$ in $S$  such that $
w_{1}e_{i_1} -w_{2}e_{i_2}=  f_{i_2}(x_{k_2}e_{i_2} -
{x_{k_3}}^{'}e_{i_3})+ \ldots+ f_{i_{t}}(x_{k_{t}}e_{i_{t}} -
{x_{k_{t+1}}}^{'}e_{i_1})= \psi (\sum_{j=2} ^{t} f_{i_j}g_{i_j})$.
Remark \ref{001} implies that $w_1=hx_{F(u_{i_2}) \setminus
F(u_{i_1})}=hx$ and $w_{2}=hx_{F(u_{i_1}) \setminus F(u_{i_2})}=hy$.
Therefore, we have
\begin{center}
$ h(xe_{i_1} -ye_{i_2})=w_{1}e_{i_1} -w_{2}e_{i_2}$.
\end{center}
This implies that
 $ h \psi (g) =\psi (\sum_{j=2} ^{t} f_{i_j}g_{i_j})$ and, hence,
 $(hg- \sum_{j=2} ^{t} f_{i_j}g_{i_j})\in\ker \psi$. Since $g \neq g_{i_j}$
 for all $1 \leq j \leq r$ one has $(hg- \sum_{j=2} ^{t} f_{i_j}g_{i_j})\neq 0$
\end{proof}

\begin{Remark}\label{3} {\em
Let  $  w $ be an element of a
minimal set of generators of $\ker(\psi)$. If $ w=\sum h_i g_i $,
where $g_i$ is a basis element of $F_1$ and $0 \neq h_i \in S$, then $h_i$ is a monomial.
Without loss of generality, we may assume that  $\psi( g_1)={t_1}^{'}
e_1-t_2 e_2$. Let $u \in \supp (h_1)$ be a monomial. Since $ ut_2e_2$ must be eliminated,
there exist a basis element  $g_j$ of $F_1$ such that $\psi( g_j)=({t_2}^{'} e_2 -t_3 e_l)$.
Without loss of generality,  we may assume $j=2$ and $l=3$. Hence,
$t_2 \frac{u}{{t_2}^{'}}=u^{'}\in \supp(h_2)$.  Again since $u't_3 e_3$
must be eliminated, without loss of generality,  we may assume there
exist a basis element  $g_3$ of $F_1$ such that $\psi( g_3)=({t_3}^{'} e_3 -t_4 e_4)$.
Therefore  $t_3\frac{u^{'}}{{t_3}^{'}}=u^{''} \in \supp(h_3)$.
Continuing these procedures yields $\psi( g_l)=({t_l}^{'} e_l -t_1
e_1)$ and $t_l\frac{u^{l-2}}{{t_l}^{'}}=u^{l-1} \in \supp(h_l)$. Hence
we obtain a cycle in $G_I$ in this way. Now if there exist another
monomial $v\in \supp(h_1)$ with $u \neq v$, then by the similar argument one can
find a new cycle in $G_I$. Hence, Lemma
\ref{004} implies that  $w$ is a combination of some other elements of
$\ker(\psi)$, a contradiction. So $h_i$ is a monomial.}
\end{Remark}

\begin{Lemma} \label{005}{\em
Let $I$, $\varphi$ and $\psi$ be as mention in  above. If  $\ker
\varphi$ is generated by linear forms, then corresponding to every
element in a minimal set of generators of $\ker( \psi)$ there is a
%chordless
cycle in $G_{I}$.}
\end{Lemma}
\begin{proof}
Let  $\sum_{i=1} ^{n}  h_{i}g_{i}$ an element of a minimal
generating set of $\ker (\psi)$. Then $\psi (\sum_{i=1} ^{n}
h_{i}g_{i} ) \\ =\sum_{i=1}^{n}  h_{i}\psi (g_{i})=0 $, where $g_i$
is a basis element of $F_1$ and   $h_i$ is monomial
for $i=1,\ldots.n$. Then $-h_{1}\psi (g_{1})=\sum_{i=2}^{n}
h_{i}\psi (g_{i}) $. Assume that $ \psi (g_{1})=x_{i_1}{e_{i_1}}-
x_{i_2}e_{i_2}$. So  $ u_{i_1}, u_{i_2} $ is a path in $G_I$.

The left-hand side of above  equation is of the form $
w_{i_1}e_{i_1} -w_{i_2}e_{i_2}$. By proof of Lemma \ref{003},  the
right-hand side of the above equation is of the form
\begin{center}
$ f_{i_2}(x_{k_2}e_{i_2} - {x_{k_3}}^{'}e_{i_3})+
f_{i_3}(x_{k_3}e_{i_3} - {x_{k_4}}^{'}e_{i_4})+ \ldots+
f_{i_t}(x_{k_t}e_{i_t}-{x_{k_{t+1}}}^{'}e_{i_1})$,
\end{center}
where $ e_{i_t}\neq e_{i_2}$ . If $ e_{i_t}= e_{i_2}$, then $
{x_{k_{t+1}}}^{'}=x_{i_1}$ and $ x_{k_t}=x_{i_2}$. Hence,
$g_{1}$ appears  in the right-hand side of equation , a contradiction.
Thus  $  u_{i_2}, u_{i_3}, \ldots, u_{i_t}, u_{i_1}$ is a path
which is different from path $ u_{i_1}, u_{i_2}$.
%$u_{i_1}, u_{i_2}, u_{i_3}, \ldots, u_{i_t}, u_{i_1}.$
\end{proof}

\begin{Theorem} \label{02} {\em
Let $I\subset S$ be a squarefree monomial  ideal
such that $G_I \cong C_m $, $m\geq
4$. Then the following conditions are equivalent:

\begin{itemize}
\item[(a)]
$I$ has a linear  resolution;

\item[(b)]
$m=n$ and with a suitable relabeling of variables for all $j$ one has $x_i\mid u_j$ for all $i$, $i+1\neq j$ and $i\neq j$, where
$n+1=1$;

\item[(c)]
$ I $ is variable-decomposable ideal;

\item[(d)]
$I$ has  linear quotients.
\end{itemize}}
\end{Theorem}

\begin{proof}
$(a) \implies (b):$
Assume that $ I $ has a linear resolution. Since $G_I$ is a cycle, by Lemma \ref{004} and Lemma \ref{005}, $\ker( \psi)=<w>$.  Let $w= \sum_{i=1}^m h_ig_i$ where $g_i$
is a basis element of $F_1$ and $h_i$ is a monomial in $S$ for $i=1,
\ldots, m$. Without loss of generality, we may assume that $G_I=u_1, u_2,\ldots, u_m, u_1$. Then
\begin{center}
$\psi(w)=\sum_{i=1}^m  h_i \psi(g_i)=h_1(x_{t_1} e_1-{x_{t_2}}^{'}
e_2)+ h_2(x_{t_2} e_2-{x_{t_3}}^{'}e_3)+ \ldots+ h_m(x_{t_m} e_m -
{x_{t_1}}^{'} e_1)=0$.
\end{center}
Therefore,  $h_1x_{t_1} e_1=h_m {x_{t_1}}^{'}e_1$. Since,
$I$ has d-linear resolution and $\deg( e_i)=d$, we conclude that $\deg(h_i)=1$ for
all $i$. Consequently,   $h_1={x_{t_1}}^{'}  $ and
$h_m=x_{t_1} $. By  similar argument $ h_j
={x_{t_j}}^{'}$ and  $ h_j =x_{t_{j+1}}$. Hence, $ x_{t_{j+1}}={x_{t_j
}}^{'}$ for all $ 1 \leq j \leq m-1 $. So $\ker(\varphi)$ is minimally generated by the following linear forms.

$ (x_{t_m} e_1-x_{t_2} e_2), (x_{t_1} e_2-x_{t_3} e_3),
(x_{t_2} e_3-x_{t_4} e_4), \ldots, \\(x_{t_{m-2}} e_{m-1}-x_{t_m}
e_m), (x_{t_{m-1}} e_m-x_{t_1} e_1).
$

For an arbitrary variable $x_i$  in $S$ there exits  $u_i $ and
$u_j$  in $ G(I)$ such that $x_i \mid u_i$ and $x_i \nmid u_j$.  Hence, by
Remark \ref{4} $x_i \in \{x_{t_1}, x_{t_2}, \ldots, x_{t_m} \}$.
It is clear  that the variables $x_{t_1}, x_{t_2}, \ldots, x_{t_m}$ are distinct and, hence, $n=m$.

Set $ x_{t_{-1}}=x_{t_{m-1}},
x_{t_{m+1}}=x_{t_1}$, $e_{0}=e_{m}$ and $e_{m+1}=e_{1}$. For $ 1 \leq i \leq m-1 $, we have
$\varphi(x_{t_{i-2}} e_{i-1}-x_{t_{i}} e_{i})=0$ and, hence, $x_{t_i} \mid u_{i-1}$ and  $x_{t_i} \nmid u_i$.
 Also, from $\varphi(x_{t_{i}} e_{i+1}-x_{t_{i+2}} e_{i+2})=0$,
we have $x_{t_i} \nmid u_{i+1} $ and $x_{t_i}
\mid u_{i+2} $.  By Remark \ref{4}  $x_{t_i} \mid u_j $ for $ j\neq i, i+1$.

$(b) \implies (c):$ It is easy to see that $I_{x_1}=<u_1, u_2>$   is  variable-decomposable,  $I^{x_1}=<u_3, \ldots, u_n>$
and  $u=x_1$ is a shedding variable.
Also, it  is clear that $x_2$ is a
shedding variable  for $I^{x_1}$ and    $(I^{x_1})^{x_2}=<u_4, \ldots,
u_n>$,  $(I^{x_1})_{x_2}=<u_3>$. Continuing these procedures yields
that  $I^{x_1}$ is variable-decomposable. Hence, $ I $
is variable-decomposable ideal.

$(c) \implies (d)$ follows by Theorem \ref{decqou}.

$(d) \implies (a)$ follows by Proposition \ref{qouresu}.
\end{proof}

\begin{Corollary}{\em
Let $I\subset S$ be a squarefree monomial ideal
generated in degree 2 and assume that $G_I \cong C_m$, $m \geq 4$.
Then $I$ has a linear resolution if and only if  $m=4$.}
\end{Corollary}

\begin{Remark}{\em
Let $I$ be a squarefree monomial ideal. If $G_I \cong C_3$, then
$I$ has linear quotients. Hence $I$ has a linear
resolution.}
\end{Remark}

Let $I$ be a squarefree monomial ideal generated in degree $2$. We
may assume that $I=I(G)$ is the edge ideal of a graph $G$. Hence, by Fr\"{o}berg's result,
$I(G)$ has a linear resolution if and only if $\bar G$ is a chordal
graph. If $G \cong C_m$, then $\bar G$ is chordal  if and only if
$m=3$ or $m=4$.  In this situation $G \cong C_m$ if and only if $G_I \cong C_m$. Hence, in this case our result  is coincide to Fr\"{o}berg's result.

\begin{Corollary}{\em
Let $I$ be a squarefree monomial  ideal generated in degree $d$ where  $G_I \cong C_m $. If $ d+2<n $ or $m\neq n$,  then $I$ can
not has a d-linear  resolution.}
\end{Corollary}

\begin{Example}{\em
Consider  monomial ideal
 $I=<xy, zy, zq, qx>\subset k[x,y,z,q]$. The graph $ G_I $ is $4$-cycle. Since $d=2$, $n=4$ and $ d+2= n $,
 $I$ has a 2-linear  resolution.

 $$0 \longrightarrow S(-4) \longrightarrow S(-3)^4 \longrightarrow S(-2)^4 \longrightarrow  I\longrightarrow 0 .$$}
\end{Example}

\begin{Example}{\em
For monomial ideal
 $I=<xyz, yzq, zqw, qwe, wex, xye> \subset  k[x,y,z,q,e,w]$, we have  $ G_I\cong C_6$. Therefore  $I$  has not a 3-linear  resolution,
 since $d=3$, $n=6$ and $ d+2< n $. The resolution of $I$ is:
 $$0 \longrightarrow S(-6) \longrightarrow S(-4)^6 \longrightarrow S(-3)^6 \longrightarrow  I\longrightarrow 0 .$$}
\end{Example}

\section{Linear resolution of monomial ideals whose $G_I$ is a tree}
\label{sec:2}
Let $I$ be a squarefree monomial ideal such that  $G_I$ is a tree.
In this section we study linear resolution of such monomial ideals.  We know that each line is a tree, therefore first we consider the following:

\begin{Proposition}\label{line}{\em
 Let $I=<u_1,\ldots, u_m>$ be a squarefree monomial  ideal generated in  degree $d$. If  $G_I=u_1, u_2, \ldots, u_m$
 is a line, then the following conditions are equivalent:

\begin{itemize}
\item[(a)]
$I$ has a linear  resolution;

\item[(b)]
For any $1\leq j\leq k \leq i \leq m$

\begin{center}
$F(u_k) \subseteq F(u_i) \cup F(u_j)$;
\end{center}

\item[(c)]
$ I $ is variable-decomposable ideal;

\item[(d)]
$I$ has  linear quotients.
\end{itemize}}
\end{Proposition}
\begin{proof}

$(a) \implies (b):$ Suppose, on the contrary, there exist $1\leq j < k < i\leq m $ and $ l
\in F(u_k) $ such that $l \notin F(u_i) \cup F(u_j)$. Since  $I$
  has a linear  resolution, we have  $ x_{F(u_i) \setminus  F(u_j)} e_j-x_{F(u_j) \setminus  F(u_i)}e_i
= f_{i}(x_{k_1}e_{i} - {x_{k_2}}^{'}e_{i+1})+
f_{i+1}(x_{k_2}e_{i+1} - {x_{k_3}}^{'}e_{i+2})+ \ldots+
f_{j-1}(x_{k_{j-1}}e_{j-1} - {x_{k_t}}^{'}e_{j})$.
 By Remark \ref{16}, $x_l \mid x_{F(u_j) \setminus  F(u_i)}$ and $x_l \mid x_{F(u_i) \setminus  F(u_j)}$ which is a contradiction.

$(b) \implies (c):$ Let $F(u_2)\setminus F(u_1)=\{l\}$. From the facts that  $F(u_2) \subseteq
F(u_1) \cup F(u_i)$,  $ l \in F(u_2)$ and $u_2 :u_1=x_l$, we conclude that $ l \in F(u_i)$ for all $2\leq i\leq m$, $I_{x_l}=<u_1>$ and $x_1$ is a shedding.  By induction on $m$,
$I^{x_l}$ is variable-decomposable, since  $I^{x_l}$ in a line of length $m-1$.

 $(c) \implies (d)$  follows by Theorem \ref{decqou}.

$(d) \implies (a)$ follows by Proposition \ref{qouresu}.

\end{proof}

\begin{Theorem} \label{01}{\em
If $I$ is a squarefree monomial  ideal which has linear
relations, then  $G_{I}$ is a tree  if and only if $\projdim(I)=1.$}
\end{Theorem}
\begin{proof}
If $G_{I}$ is a tree, then $G_{I}$ has no cycle. Therefore by Lemma
\ref{005}, $\ker(\psi)=0$. Hence the linear  resolution of $I$
is of the form
\begin{center}
$0 \longrightarrow F_{1} \longrightarrow F_{0}\longrightarrow
I\longrightarrow 0 $
\end{center}
and $\projdim(I)=1$.

Conversely, assume that  $\projdim(I)=1$. Then $\ker(\psi)=0$ and by Lemma \ref{004}
$G_{I}$ has no cycle. Since $I$ has  linear relations, by Lemma \ref{003}, $G_{I}$ is a connected
graph.  Therefore $G_{I}$ is a tree.
\end{proof}

\begin{Proposition}{\label{pd=1}}{\em
Let $I$ be a squarefree monomial  ideal with  $\projdim(I)=1$. Then $I$ has a linear  resolution if and only if $G_{I}$ is a connected graph.}
\end{Proposition}
\begin{proof}
Assume that $G_{I}$ is a connected graph. Since $\projdim(I)=1$, Lemma \ref{004} implies that $G_{I}$ has no cycle and, hence, it is a tree.  So it is enough to show $I$ has linear relations.  For $u_i, u_j \in G(I)$ there exist a unique  path between $u_i$ and $u_j$ in $G(I)$. Assume that $$w e_i -w^{'} e_j =f_{i_1}(x_{k_1}e_{i_1} - {x_{k_2}}^{'}e_{i_2})+
f_{i_2}(x_{k_2}e_{i_2} - {x_{k_3}}^{'}e_{i_3})+ \ldots+
f_{i_{t-1}}(x_{k_{t-1}}e_{i_{t-1}} - {x_{k_{t}}}^{'}e_{i_t})$$
be an element of $\ker(\varphi)$ which is
obtained from this path. If $w e_i -w^{'} e_j =x_{F(u_j) \setminus  F(u_i)}e_i-x_{F(u_i) \setminus  F(u_j)}e_j$, we are done.  So assume that the equality does not holds. Then $x_{F(u_j) \setminus  F(u_i)}e_i-x_{F(u_i) \setminus  F(u_j)}e_j $ is a minimal element in $\ker(\varphi)$. Hence, there exists $g \in F_1$ such that $\psi(g)=x_{F(u_j) \setminus  F(u_i)}e_i-x_{F(u_i) \setminus  F(u_j)}e_j $. Remark \ref{001} implies that there exists a monomial  $h \in S$ such that $h\psi(g)=w e_i -w^{'} e_j=f_{i_1}\psi(g_{i_2})+ \ldots +f_{i_{t-1}}\psi(g_{i_{t-1}})$. Therefore  $\psi(hg-f_{i_1}g_{i_2}-\cdots -f_{i_{t-1}}g_{i_{t-1}})=0$ and  $hg-f_{i_1}g_{i_2}-\cdots -f_{i_{t-1}}g_{i_{t-1}}\neq 0$, a contradiction.

The converse follows from Lemma \ref{003}.
\end{proof}

\begin{Proposition}\label{08}{\em
Let $I=<u_1,\ldots, u_m>$ be a squarefree monomial  ideal generated in degree $d$ which has  linear
quotients. Assume that
$G_I$ is a tree and $ v $ is a monomial in degree $d$ which is  a
leaf in $G_{<I,v>}$. Then the following
conditions are equivalent:

\begin{itemize}
\item[(a)]
$<I,v>$ has a linear  resolution;

\item[(b)] Let $u_i$ be the branch of $v$ and
 $ F(u_i) \setminus F(v)=\{l\}$, then $l \in \bigcap_{t=1}^m F(u_t)$;

\item[(c)]
$<I,v>$ has linear quotients.
\end{itemize}}
\end{Proposition}
\begin{proof}

$(a) \implies (b):$
Suppose, on the contrary, that there exist a $1\leq j\leq m$ such that $l \notin
F(u_j)$. Let $v, u_i=u_{i_1}, u_{i_2}, \ldots, u_{i_{t-1}}, u_{i_t}=u_j$ be the unique  path between $v$ and
$u_j$. Without loss of generality, we  may assume that $ l \in F(u_{i_r}) $
for all $r$, $1 \leq r \leq t-1$. Since $<I,v>$ has a linear resolution, we have
 $x_{F(u_j) \setminus  F(v)} e_v-x_{F(v) \setminus  F(u_j)}e_j =f_{0}(x_{i_0}e_{v} -
x_{{i_1}^{'}} e_{i_1}) + f_{1}(x_{i_1} e_{i_1} -
x_{{i_2}^{'}}e_{i_2}) + \ldots+ f_{t-1}(x_{i_{t-1}} e_{i_{t-1}} -
x_{{i_t}^{'}}e_{t}).$
 By Remark \ref{16}, we know that $x_l \mid x_{F(u_j) \setminus  F(v)}$ and $x_l \mid x_{F(v) \setminus  F(u_j)}$,
 this is a contradiction .

$(b) \implies (c):$  We now that there is an
admissible order  $ v_1, v_2, \ldots, v_m $ of $G(I)$. Since by our assumption  $\{l\}=F(u_i) \setminus
F(v)$ and  $l \in F(u_j) $ for any
$1\leq j \leq m$, we conclude that the order $ v_1, v_2, \ldots, v_m , v$ is an admissible order
for $<I,v>$.

$(c) \implies (a)$  follows from Proposition \ref{qouresu}.
\end{proof}

\begin{Proposition}\label{IL}{\em
 Let $I=<u_1,\ldots, u_m>$ be a squarefree monomial  ideal generated in
 degree $d$.  If  $G_I$ is a tree, then $I$ has a linear  resolution if and only if $L$ has a linear  resolution for all $L \subseteq I$, where $G(L)\subset G(I)$ and $G_L$ is a line.}
\end{Proposition}
\begin{proof} Assume that $I$ has a linear resolution. Since $G_I$ is a tree, we have $\projdim (I)=1$. So if $L\subset I$ with $G(L)\subset G(I)$ and $G_L$ is a line, then $L$ has linear relations and $\projdim(L)=1$. Therefore $L$ has a linear resolution.

For the converse, by  our assumption there exists  a monomial ideal $J_0\subset I $ such that $G(J_0)=\{u_{i_1}, \ldots, u_{i_t}\}\subset G(I)$, $G_{J_0}$ is a line and has linear resolution. Therefore  $J_0$ has linear quotients.  Take $v \in
V(G_I) \setminus V(G_{J_0})$ such that $v$ and $u_{i_j}$ are adjacent in
$G_I$ for some $1 \leq j \leq t$. Set $F(u_{i_j})
\setminus F(v)=\{l\}$. Since $J_0$ has linear quotients there exist a path
between $u_{i_r}$ and $u_{i_j}$ for all $1 \leq r \leq t$. Therefore we have line $u_{i_r}, \ldots,
u_{i_j}, v $ in $G_I$. By our hypothesis $L=<u_{i_r}, \ldots, u_{i_j}, v>$ has a
linear resolution and  Proposition \ref{line} implies that $ F(u_{i_j})
\subseteq F(v) \cup F(u_{i_r})$.  Therefor $\{l\} \in F(u_{i_r})$
and Proposition \ref{08} implies that $J_1=<J_0, v>$ has linear
quotients. Now replace $J_0$ by $J_1$ and do the same procedure until we obtain $I$.
\end{proof}

\begin{Theorem}\label{tree}{\em
Let $I$ be a squarefree monomial ideal which is generated in one degree. If $G_I$ is a tree, then the following conditions are equivalent:
\begin{itemize}
\item[(a)]
$I$ has a linear resolution;
\item[(b)]
$I$ has  linear relations;
\item[(c)]
$G_I^{(u,v)}$ is a connected graph for all $u$, and $v$ in $G(I)$;
\item[(d)]
If $u=u_1,u_2,\ldots,u_s=v$ is the unique path  between $u$ and $v$ in $G_I$, then $F(u_j)\subset F(u_i)\cup F(u_k)$  for all $1\leq i\leq j\leq k \leq s$;
\item[(e)]
 $L$ has a linear  resolution for all $L \subseteq I$, where $G(L)\subset G(I)$ and $G_L$ is a line
\end{itemize}}
\end{Theorem}
\begin{proof}
$(a)\implies (b) $  is trivial.
$(b)\implies (c)$ follows from Proposition \ref{009}.
$(c)\implies (d)$ :  for all $1\leq i\leq j \leq k\leq s$, $G_I^{(u_i,u_k)}$ is connected and $u_j$ is a vertex of this graph. Therefore $F(u_j)\subset F(u_i)\cup F(u_k)$.
$(d)\implies (e)$ follows from Proposition \ref{line}.
$(e)\implies (a)$ follows from Proposition \ref{IL}.
\end{proof}
\begin{Theorem}\label{019}{\em
 Let $I$ be a squarefree monomial  ideal generated in degree $d$. If $G_I$ is a tree,
 then the following are equivalent:

\begin{itemize}
\item[(a)]
$I$ has a linear  resolution;

\item[(b)]
$ I $ is variable-decomposable ideal;

\item[(c)]
$I$ has linear quotients.
\end{itemize}}

\end{Theorem}

\begin{proof}
$(a) \implies (b):$  We know that $\projdim(I)=1$, since $G_I$ is a tree and $I$ has a linear resolution. With out loss of generality we may assume  that
$u_1$ is a vertex of degree one in $G_I$ and  $u_2$ be the unique neighborhood of $u_1$  in $G_I$.  Set $F(u_2)
\setminus F(u_1)=\{ l \}$.   Proposition
\ref{009} implies that $G_I^{(u_1,u_i)}$ is a connected graph for all $u_i$.
If $l\not\in F(u_i)$ for some $i>2$, then $F(u_2)\nsubseteq F(u_1) \cup F(u_i)$ and $u_2 \notin
V(G^{(u_1,u_i)}_I )$.  Therefore
$G^{(u_1,u_i)}_I $ is not connected, a
contradictions. Hence $I_{x_l}=\{u_1\}$  and $G(I^{x_l})=G(I)\setminus \{u_1\}$. It is easy to see that $x_l$ is a shedding variable. Since  $
G_{I^{x_l}}$ is a tree and has linear relations,  by
induction on $|G(I)|$, we conclude that $I^{x_l}$ is
variable-decomposable.  Therefore $ I $ is variable-decomposable
ideal.
$(b) \implies (c)$ follows by Theorem \ref{decqou}.
$(c) \implies (a)$ follows by Proposition \ref{qouresu}.
\end{proof}
\begin{Remark}{\em
In Theorem \ref{019}, we show that if $G_I$ is a tree and $I$ has a linear resolution, then $I$ has  linear quotients. In the following we present  an admissible order for $I$ in this case. we choose  order $u_{r_1}, \ldots, u_{r_m}$ for the elements of $G(I)$
such that the subgraph on vertices $\{u_{r_1}, \ldots, u_{r_t}\}$ is a connected graph for  $ 1 \leq t \leq m$.
We show tat this order is an admissible order. If this order is not an admissible order, then there exists a $j < i$ such that
for all $k <i$ with $ F(u_{r_k})\setminus F(u_{r_i})=\{l\}$, we have $l \notin F(u_{r_j})$.
Since there is a path  $u_{r_j}, \ldots, u_{r_k}, u_{r_i}$
Remark \ref{16} implies  that $x_l \mid x_{F(u_{r_i}) \setminus  F(u_{r_j})}$ and $x_l \mid x_{F(u_{r_j}) \setminus  F(u_{r_i})}$, a contradiction.}
\end{Remark}

A simplicial complex $ \Delta $ over a set of vertices
$[n] = \{1, \ldots, n\}$ is a collection of subsets of $[n]$  with the property
that $\{i\}\in \Delta $ for all $i$ and if $F \in \Delta$, then all subsets
of $F$ are also in $\Delta$. An element of $\Delta$ is called a face and the dimension of a face $F$ is defined as $|F|-1$, where $|F|$ is the number of vertices of $F$.
The maximal faces of $\Delta$ under inclusion are called facets  and the set of all facets  denoted by $\mathcal{F}(\Delta)$. The dimension of the simplicial complex $\Delta$ is the
maximal dimension of its facets. A subcollection of $\Delta$ is a simplicial complex whose
facets are also facets of $\Delta$. In other words a simplicial complex generated by a
subset of the set of facets of $\Delta$. Let $\Delta$ be a simplicial complex on $[n]$ of dimension $d-1$.
For each $0 \leq i \leq d-1$  the $i$th skeleton of $\Delta$ is the
simplicial complex  $\Delta^{(i)}$ on $[n]$ whose faces are those
faces $F$ of $\Delta$ with $\mid F\mid \leq i+1$.
 We say that a simplicial complex $\Delta$ is
connected if for facets $F$ and $G$ of $\Delta$ there exists a sequence of facets $F = F_0, F_1,
\ldots, F_{q-1}, F_q = G$ such that $F_i \cap F_{i+1} \neq
\emptyset$ for $i=0, \ldots, q-1$. Observe that $\Delta$ is
connected if and only if $\Delta^{(1)} $ is
connected.

Let $\Delta$ be a simplicial complex on $[n]$.
The Stanley-Reisner ideal of $\Delta$ is a squarefree monomial
ideal $I_\Delta=<x_{i_1}\ldots x_{i_p} \mid \{x_{i_1}, \ldots, x_{i_p}\}
\notin \Delta >$. Conversely, let $I \subseteq k[x_1, \ldots, x_n]$ be a squarefree monomial ideal.
The Stanley-Reisner complex of $I$ is the simplicial
complex $\Delta $ on $[n]$ such that $I_\Delta=I$.
The Alexander dual of $\Delta$ is the simplicial complex
$\Delta^{\vee}=<\{x_1, \ldots, x_n\} \setminus F \mid F \notin
\Delta >.$

\begin{Definition}{\em \cite{SF}
Let $ \Delta$ be a simplicial complex. A facet $F \in \mathcal{F}( \Delta)$ is
said to be a leaf of $ \Delta$  if either $F$ is the only facet of $ \Delta$ or
there exists a facet $G \in  \mathcal{F}( \Delta)$ with $G \neq F$, called a branch $F$,
such that $H \cap F  \subseteq G \cap F $ for all
$H \in \mathcal{F}( \Delta)$ with $H \neq F$. A connected simplicial complex $ \Delta$
is a tree if every nonempty subcollection of $ \Delta$ has a leaf. If $ \Delta$ is
not necessarily  connected, but every subcollection has a leaf, then $ \Delta$ is  called a
forest.}
\end{Definition}
If $\Delta$  is a simplicial tree, then we can always order the
facets $F_1, \ldots, F_q$ of $\Delta$ such that $F_i$ is a leaf of
the induced subcomplex $<F_1, \ldots, F_{i-1}>$. Such an ordering on
the facets of $\Delta$  is called a leaf order.
A simplicial complex $\Delta$ is a
quasi-forest if $\Delta$ has a leaf order. A connected quasi-forest
is called a quasi-tree.

Consider an arbitrary monomial ideal $I = <u_1, \ldots, u_m >$.
For any subset $\sigma$ of $\{1, \ldots, m \}$, we write
$u_{\sigma}$ for the least common multiple of $\{u_i \mid i \in
\sigma \}$ and set $\a_{\sigma} = \deg u_{\sigma}$.  Let $G(I)=\{u_1, \ldots, u_m\}$. The Scarf complex $ \Delta_I $ is the
collection of all subsets of $\sigma\subset\{1, \ldots, m\}$ such that $u_{\sigma}$ is unique.
As first noted by Diane Taylor \cite{T}, given a monomial  ideal $I$
in a polynomial ring $S$ minimally generated by monomials $u_1,
\dots, u_m$, a free resolution of $I$ can be given by the simplicial
chain complex of a simplex with $m$ vertices. Most often Taylor's
resolution is not minimal. The Taylor complex $F_{\Delta_I}$
supported on the Scarf complex $\Delta_I$ is called the algebraic
Scarf complex of the monomial ideal $I$.
 For  more information about Taylor
complex we refer to \cite{MS}.

The following results  will be used later.
\begin{Lemma} \label{0017}\cite{MS}{\em
If $I$ is a monomial ideal in $S$, then every free resolution of
$S/I$ contains the algebraic Scarf complex $F_{\Delta_I}$ as a
subcomplex.}
\end{Lemma}

%\begin{Theorem} \label{09}
%(\cite{Pv}, \cite{Ph}). Let $ \Delta$ be a simplicial complex on $r$
%vertices.

%$i)$  $\Delta$ is the Scarf complex of a monomial ideal if and only
%if $ \Delta$ is not the boundary of a simplex on $r$ vertices.

%$ii)$ $ \Delta$ minimally resolves a monomial ideal if and only if $
%\Delta$ is acyclic.
%\end{Theorem}

\begin{Proposition}\cite[Corollary 4.7]{SF2} \label{15}{\em
 Every simplicial tree is the Scarf complex of a monomial ideal $I$ and
  supports a minimal resolution of $I$.}
\end{Proposition}

In \cite{FH} Faridi and Hersey studied minimal free resolution of
squarefree monomial ideals with projective dimension $1$.  They prove the following.

\begin{Theorem} \label{010}{\em
 Let I be a squarefree monomial ideal in a polynomial ring S and $\Delta $ be a simplicial complex such that $I=I_\Delta$.
 Then the following statements are equivalent:

\begin{itemize}
\item[(a)]
$\projdim I \leq 1$;

\item[(b)]
 $\Delta^{\vee}$ is a quasi-forest;

\item[(c)]
$ S/I$ has a minimal free resolution supported on a graph-tree.
\end{itemize}}
\end{Theorem}

If $\Delta$ is a simplicial complex and $ \dim \Delta=1$,
then the geometric realization of $\Delta$ is a graph. In this situation we say
$\Delta$ is a graph.

\begin{Lemma} \label{0010}{\em
Let $I$ be a monomial ideal. Set    $G=
\Delta_I^{(1)}$. If $G_I$ is a $c_3$-free graph, then $G_I$ is a subgraph
of $G$.}
\end{Lemma}
\begin{proof}
One has $V( G)=V(G_I)=G(I)$. Let $\{ u_{i}, u_{j}\} \in E( G_I)$. Assume that $\{ u_{i}, u_{j}\}$ is not an edge in $G$.
Then there exits $\sigma\subset\{1, \ldots, m\}$ such that $\{i,j\} \neq \sigma$ and
 $u_{\{i,j\}}=u_{\sigma}$. Let $r \in \sigma \setminus
\{i,j\}$,  then  $F(u_r) \subseteq F(u_i) \cup F(u_j) $. We may assume
that $F(u_i)=A \cup \{i\}$ and $F(u_j)=A \cup \{j\}$. Hence
$F(u_r) \subseteq A \cup \{i, j\}$. Therefore  $\{
u_{i}, u_{r}\}$ and $\{ u_{j}, u_{r}\}$ are in $E( G_I)$, which is a
contradiction.
\end{proof}

\begin{Remark} \label{14}{\em
Let $I$ be a monomial ideal which has a linear resolution and $u_i, u_j\in G(I)$. Assume that
$u_i=u_{i_1},..,u_{i_{t-1}},u_{i_t}=u_j$ is a path  between $u_i$ and
$u_j$. Then  $ x_{F(u_j)\setminus
F(u_i)}e_{i}-x_{F(u_i)\setminus F(u_j)}e_{j} \in \ker(\varphi)$
and is a linear combination of linear forms which comes from the given path. By Remark \ref{4} and Remark \ref{16},
$F(u_{i_r}) \subseteq F(u_i) \bigcup F(u_j)$, for all $i\leq r\leq j$.
}
\end{Remark}

\begin{Theorem} \label{011}{\em
 Assume that  $I$ is a squarefree monomial  ideal generated in degree $d $. If $G_I$ is a tree, then $I$
has a linear resolution if and only if $G_I \cong  \Delta_I $.}
\end{Theorem}
\begin{proof}
Assume that $I$ has a linear resolution.
By Theorem \ref{01}  $ \projdim I \leq 1$ and by
Lemma \ref{0017}  $ \dim \Delta_I=1$. By Lemma \ref{0010}  $ G_I$ is a subgraph of  $\Delta_I $. Now let $ \{ u_{r},
u_{t}\} \in \Delta_I $.
Suppose that $\{ u_{r}, u_{t}\} \notin E( G_I)$. Since $x_{F(u_t)\setminus F(u_r)}e_{r}-x_{F(u_r)\setminus F(u_t)}e_{t} \in  \ker(\varphi)  $ and $I$ ha a linear resolution, we have
 $$ x_{F(u_t)\setminus F(u_r)}e_{r}-x_{F(u_r)\setminus F(u_t)}e_{t}=f_r(x_{k_r} e_r -
x_{{k_{r+1}}^{'}}e_{i_{r+1}})+  \ldots + f_{j}(x_{k_{j}}e_{i_{j}} -
x_{{k_t}^{'}} e_{t}).$$
Set $\sigma=\{r, i_{r+1}, \ldots, i_j, t\} $. By  Remark \ref{14}  $ m_\sigma=m_{\{r, t\}}$, which is a contradiction.

Conversely  assume that $G_I \cong  \Delta_I $. Therefore $\Delta_I$ ia a tree.
By Proposition \ref{15}  $\Delta_I$ supports a
minimal free  resolution of $I$. Therefore by Theorem \ref{010}  $ \projdim I
\leq 1$. If $ \projdim I = 0$, then $I$ is
principal monomial ideal and, hence,
 $I$ has a linear resolution.

 Let $ \projdim I = 1$. If $I$ has not a linear resolution, then there exists
 $ x_{F(u_j)\setminus F(u_i)}e_{i}-x_{F(u_i)\setminus F(u_j)}e_{j}\in  \ker(\varphi)$
such that this element belong to a minimal set of generators of $\ker(\varphi)$ and
 $\deg ( x_{F(u_i)\setminus F(u_j)}) \geq 2$.
There exists a unique path  $u_{i}=u_{i_1}, u_{i_2}, \ldots, u_{i_{t-1}},
u_{i_{t}}=u_{j}$ between $u_i$ and $u_j$ in $G_I$. By Lemma \ref{002} and Remark \ref{001}
there exists a monomial $w$ such that\\ $w(x_{F(u_j)\setminus F(u_i)} e_{i}-x_{F(u_i)\setminus
F(u_j)}e_{j})=f_{i_1}(x_{k_1}e_{i_1} - {x_{k_2}}^{'}e_{i_2})+
f_{i_2}(x_{k_2}e_{i_2} - {x_{k_3}}^{'}e_{i_3})+ \ldots+
f_{i_{t-1}}(x_{k_{t-1}}e_{i_{t-1}} - {x_{k_t}}^{'}e_{i_t})$.\\
Set $\psi (g)= x_{F(u_j)\setminus F(u_i)}e_{i}-x_{F(u_i)\setminus
F(u_j)}e_{j},\; \psi(g_1)=x_{k_1}e_{i_1} - {x_{k_2}}^{'}e_{i_2},
\ldots,\; \psi(g_{t-1})\\=x_{k_{t-1}}e_{i_{t-1}}
-{x_{k_t}}^{'}e_{i_t}$. Then
$ \psi(\Sigma_{r=1}^{t-1}f_{i_r}g_r-wg)=0$.  Since $\Sigma_{r=1}^{t-1}f_{i_r}g_r-wg\neq 0$, one has $\ker( \psi) \neq
0$, which is a contradiction.
\end{proof}
\section{linear resolution of some classes of monomial ideals}
\label{sec:3}
In this section as applications of our results, we determine linearity of resolution for some classes of monomial ideals.

Let $I$ be a squarefree Cohen-Macaulay monomial ideal of codimension $2$ and $\Delta$ be a simplicial complex such that $I=I_{\Delta}$. In \cite{HSY} the authors showed that $\Delta$ is shellable. Moreover one can see that $\Delta$ is vertex decomposable, see \cite{AS}.  Now assume that  $I$ is  generated in one degree.  Since $\projdim(I)=1$,  as a corollary of Proposition \ref{pd=1} and Theorem \ref{019}, we have:

\begin{Corollary}\label{cm2}{\em
Let $I$ be a squarefree Cohen-Macaulay monomial ideal of codimension $2$. Then $I$ has a linear  resolution if and only if $G_{I}$ is a connected graph. Indeed in this case $G_I$ is a tree and the following conditions are equivalent:
\begin{itemize}
\item[(i)] $I$ has a linear resolutions;
\item[(ii)] $I$ has linear quotients;
\item[(iii)] $I$ is variable decomposable.
\end{itemize} }
\end{Corollary}
The following example  shows that there are Cohen-Macaulay monomial ideal of codimenstion $2$ with and without a linear resolution.
\begin{Example}{\em
$(i)$- let $I=(xy,yz,zt)\subset K[x,y,z,t]$ . Then $I$ is Cohen-Macaulay monomial ideal of codimenstion $2$ with  a linear resolutions

$(ii)$- let $I=<xy, zt> \subset K[x,y,z,t]$.
 It is easy to see that $I$ is a Cohen-Macaulay monomial ideal of codimension $2$ which has not a linear  resolution.}
\end{Example}

\begin{Remark}{\em
Let $I$ be a squarefree monomial ideal. If $G_I $ is a complete graph, then the following statements hold.
\begin{itemize}
\item[(a)]
$I$ has a linear  resolution;
\item[(b)]
$ I $ is variable-decomposable ideal;
\item[(c)]
$I$ has  linear quotients.
\end{itemize}}
\end{Remark}
In \cite{CD} Conca and De Negri introduced path ideal
of a graph. Let $G$ be a directed
graph on vertex set $\{1, \ldots, n\}$. For integer $2 \leq t \leq n$,
a sequence ${i_1}, \ldots, {i_t}$ of distinct vertices of $G$ is called a path of length $t$, if there are $t-1$ distinct directed edges $e_1, \ldots, e_{t-1}$, where $e_j$ is an edge from ${i_j}$ to ${i_{j+1}}$. The path ideal of $G$ of length
$t$ is the monomial ideal $ I_t(G) = <\prod_{j=1}^t x_{i_j} >$, where ${i_1}, \ldots, {i_t}$ is a path of length $t$ in $G$. Let $C_n$  denote the $n$-cycle on vertex set $V=\{1, \ldots, n\}$.  In \cite [proposition 4.1]{FV} it is shown that $S/I_2(C_n)$ is vertex decomposable/ shellable/ Cohen-Macaulay if and only if $n=3$ or $5$.  Saeedi, Kiani and Terai in \cite{SKT} showed that if $2<t\leq n$, then $S/I_t(C_n)$ is sequentially Cohen-Macaulay if and only if $t=n$, $t=n-1$ or $t=(n-1)/2$. In \cite{AS} it is shown that  $S/I_t(C_n)$ is Cohen-Macaulay if and only if it is shellable and if and only if  $I_t(C_n)$ is vertex decomposable.

It is easy to see that if $t<n-1$, then $G_{I_t(C_n)}\cong C_n$. Hence, by Theorem \ref{02}  $ I_t(C_n) $ has a linear  resolution if and only if $t=n-2$. For $t=n-1$, since $G_{I_t(G)}$ is a complete graph,   $ I_t(G) $ has a linear  resolution. Also, in these cases having a linear resolution is equivalent to have linear quotients and it is equivalent to variable decomposability of $I_t(C_n)$.
\begin{Corollary}\label{cn}{\em
 $ I_t(C_n) $ has a linear  resolution if and only if $t=n-2$ or $t=n-1$. Moreover the following conditions are equivalent:
\begin{itemize}
\item[(a)]
$I_t(G)$ has a linear  resolution;
\item[(b)]
$ I_t(G) $ is variable-decomposable ideal;
\item[(c)]
$I_t(G)$ has  linear quotients.
\end{itemize}}

\end{Corollary}

\begin{Corollary}\label{Ln}{\em
Let $L_n$ be a line on vertex set $\{1, \ldots, n\}$ and $ I_t(L_n)$ be the  path ideal of $L_n$. Then $ I_t(L_n) $ has a linear  resolution if and only if $t\geq \frac{n}{2}$.}
\end{Corollary}
\begin{proof}
Let $L_n=1, \ldots, n$ be a line. It is easy to see that $G_{ I_t(L_n)}\cong L_{n-t+1}$ and $I_t(L_n)=<\prod_{i=1}^t x_i, \ldots, \prod_{i=t+1}^{2t} x_i, \ldots, \prod_{i=n-t+1}^n x_i>$.
If $n-t+1 > t+1$, then $F(u_2) \nsubseteq F(u_1) \bigcup F(u_n)$. Hence Theorem \ref{05} implies that $I_t(G) $ has not a linear  resolution. If $n-t+1 \leq t+1$, i.e $t\geq \frac{n}{2}$, then it is clear that for any $1\leq j\leq k \leq i \leq m$ one has:
\begin{center}
$F(u_k) \subseteq F(u_i) \bigcup F(u_j)$.
\end{center}
Therefore, by Proposition \ref{line}, $ I_t(G) $ has a linear  resolution and the equivalent conditions hold.
\end{proof}

\section{Cohen-Macaulay  simplicial complex}
\label{sec:4}
Let $\Delta=< F_1, \ldots,  F_m>$ be a  simplicial complex on vertex set $[n]$ and $I_{\Delta}\subset k[x_1,\ldots,x_n] $ be its Stanley-Reisner ideal. For each $F\subset [n]$, we set $\bar{F_i}=[n]\setminus F_i$ and $P_F=(x_j:\; j\in F)$.  It is well known that $I_{\Delta}=\bigcap_{i=1}^{m} P_{\bar{F_i}}$ and $I_{\Delta^{\vee}}=(x_{\bar{F_i}}:\; i=1,\ldots,m)$, see \cite{HH}. The simplicial complex $\Delta$ is called pure if all facets of it have the same dimension. It is easy to see that $\Delta$ is pure if and only if $I_{\Delta^{\vee}}$ is generated in one degree.
The $k$-algebra $k[\Delta] =S/{I_{\Delta}}$ is called the Stanley-Reisner ring of $\Delta$. We
say that $\Delta$ is Cohen-Macaulay over $k$ if $k[\Delta]$ is Cohen-Macaulay.  It is  known $\Delta$ is a Cohen-Macaulay over $k$ if and only if $I_{\Delta^{\vee}}$ has a linear resolution,  see \cite{ER}.  Since every Cohen-Macaulay simplicial complex is pure, in this section, we consider only pure simplicial complexes.

The simplicial complex $\Delta$ is called shellable if its facets can be ordered
$F_1,F_2,\dots,F_m$ such that, for all $2\leq i\leq m$, the subcomplex $\langle F_1,\ldots,F_{i-1}\rangle \cap \langle F_i\rangle$ is pure of dimension $\dim(F_i)-1$.

For the simplicial complexes $\Delta_1$ and $\Delta_2$ defined on disjoint vertex sets, the join of $\Delta_1$ and $\Delta_2$ is $\Delta_1*\Delta_2=\{F\cup G\; :\; F\in\Delta_1,\; G\in\Delta_2\}$. For a face $F$ in $\Delta$, the link, deletion and star of $F$ in $\Delta$ are respectively, denoted by
$\link_\Delta F$, $\Delta\setminus F$ and $\star_\Delta F$ and are defined by
$\link_\Delta F = \{G\in \Delta \; :\; F\cap G =\varnothing,\; F\cup G\in \Delta\}$,  $\Delta\setminus F=\{G\in\Delta\;:\; F\nsubseteq G\}$ and $\star_\Delta F= \langle F\rangle*\link_\Delta F$.

A face $F$ in $\Delta$ is called a shedding face if every face $G$ of $\star_\Delta F$ satisfies the following exchange property: for every $i\in F$ there is a $j\in [n]\setminus G$ such that $(G\cup\{j\})\setminus\{i\}$ is a face of $\Delta$.  A simplicial complex $\Delta$ is recursively defined to be $k$-decomposable if either $\Delta$ is a simplex or else has a shedding face $F$ with $\dim(F)\leq k$ such that both $\Delta\setminus F$ and $\link_\Delta F$ are $k$-decomposable. $0$-decomposable simplicial complexes are called vertex decomposable.

It is clear that $x_{\bar{F_i}}$ and $x_{\bar{F_j}}$ are adjacent in $G_{I_{\Delta^{\vee}}}$ if and only if $F_i$ and $F_j$  are connected  in codimension one, i.e, $|F_i\cap F_j|=|F_i|-1$. A simplicial complex $\Delta$ is called connected in codimension one or  strongly connected if for any two facets $F$  and $G$ of $\Delta$ there exists a sequence of facets $F = F_0,  F_1, \ldots, F_{q-1}, F_q = G$ such that  $F_i$ and  $ F_{i+1}$ is connected in codimension one for each $i=1,\ldots,q-1$ . Hence we have the following:
\begin{Lemma}\label{cd1-c}{\em
A  simplicial complex $\Delta$ is connected in codimension one if and only if  $G_{I_{\Delta^{\vee}}}$ is a connected graph.}
\end{Lemma}
 For facets  $F$ and $G$  of $\Delta$, we introduced a subcomplex $\Delta^{(F,G)}=\langle L\in \mathcal{F}(\Delta):\; F\cap G\subset L\rangle$. It is easy to see that $\Delta^{(F,G)}$ is connected in codimension one if and only if $G_{I_{\Delta^{\vee}}}^{(x_{\bar F},x_{\bar G})}$ is a connected graph. Hence by Proposition \ref{009}  we have:
 \begin{Corollary}\label{delta fg}{\em
 Let $\Delta $ be a simplicial complex on vertex set $[n]$.  Then $I_{\Delta^{\vee}}$ has linear relations if and only if $\Delta^{(F,G)}$ is connected in codimension one for all facets $F$ and $G$ of $\Delta$.}
 \end{Corollary}
 Suppose that $\Delta$ is a simplicial complex of dimension $d$, i.e, $|F_i|=d+1$ for all $i$. We associate to $\Delta$ a simple
graph $G_{\Delta}$  whose vertices are
labeled by the facets of $\Delta$. Two vertices $F_i$ and $F_j$ are
adjacent if  $F_i$ and $F_j$ are connected in codimension one.  If $F_i$ and $F_j$ are adjacent in $G_{\Delta}$, then $\mid F_i \cap  F_{i+1}\mid =d$. It is easy to see that  $ \mid\bar{F_i} \cap\bar{F_j}\mid=n-d-2$. Therefor $x_{\bar{F_i}}$ and $x_{\bar{F_j}}$ is adjacent in $G_{I_{{\Delta}^{\vee}}}$ and, hence, $G_{\Delta} \cong G_{I_{{\Delta}^{\vee}}}$.

Now assume that $G_{\Delta}\cong G_{I_{{\Delta}^{\vee}}}$ is a line.  Proposition \ref{line}, implies that  $I_{{\Delta}^{\vee}}$ has a linear resolution if and only if for any $1\leq j\leq k \leq i \leq m$, $\bar{ F_k} \subseteq \bar{F_i} \bigcup \bar{F_j}$. By  Eagon-Reiner \cite{ER}, we have  the following:
\begin{Corollary}\label{d-line}{\em
Let $\Delta=< F_1, \ldots,  F_m>$ be a pure simplicial complex. If  $G_{\Delta}=F_1, F_2, \ldots, F_m$ is a line, then $\Delta $ is a Cohen-Macaulay  if and only if   $F_i \cap F_j \subseteq F_k$ for any $1\leq j\leq k \leq i \leq m$. Moreover in this case  the following conditions are equivalent:
\begin{itemize}
\item[(a)] $\Delta^{(F,G)}$ is connected in codimension one for all facets $F$ and $G$ in $\Delta$.
\item[(b)] $\Delta$ is Cohen-Macaulay.
\item[(c)] $\Delta$ is shellabe
\item[(d)] $\Delta$ is vertex decomposable simplicial complex.
\end{itemize}}
\end{Corollary}
Also  a consequence of Theorem \ref{02}, we have:
\begin{Corollary}\label{d-cycle}{\em
Let $\Delta=< F_1, \ldots,  F_m>$ be a pure simplicial complex. If $G_{\Delta} \cong C_m$, then $\Delta$ is a Cohen-Macaulay  if and only if $m=n$ and with a suitable relabeling of facets $F_i$, we have  $i \in F_i \cap F_{i+1}$ and $i \notin F_j$ for all $j \neq i, i+1$ ( $F_{m+1}=F_1$). Moreover in this case $\Delta$ is shellabe and vertex decomposable simplicial complex.}
\end{Corollary}
Again a corollary of Theorem \ref{019},  Theorem \ref{011} and Theorem \ref{tree} we have:
\begin{Corollary}\label{d-tree}{\em
Let $\Delta=< F_1, \ldots,  F_m>$ be a pure simplicial complex.  If $G_{\Delta}$ is a tree, then the following conditions are equivalent:
\begin{itemize}
\item[(a)] $\Delta^{(F,G)}$ is connected in codimension one for all facets $F$ and $G$ in $\Delta$.
\item[(b)] if $F=F_1, F_2,\ldots,F_s=G$ is a unique path in $G_{\Delta}$, then $F_i\cap F_k\subset F_j$ for all $1\leq i\leq j\leq k\leq s$.
\item[(c)] $\Delta$ is Cohen-Macaulay.
\item[(d)] $\Delta$ is shellabe
\item[(e)] $\Delta$ is vertex decomposable.
\item[(f)] $G_{\Delta} \cong \Delta_{I_{{\Delta}^{\vee}}}$.
\end{itemize}}
\end{Corollary}

%%%%%%%%%%%%%%%%%%%%%%%%%%%%%%%%%%%%%%%%%%%%%%%%%%%%%%%%%%%%%%%%%%%%%%%%%%%%%%%%%%%%%%%%%%%%%%%%%%%%%%%%%%%

\end{document}